\documentclass[12pt,oneside]{amsart}

\usepackage[utf8]{inputenc}
\usepackage[english]{babel}
\usepackage[margin=3cm]{geometry} 
\usepackage{amsmath,amsthm,amssymb}
\usepackage{mathtools}
\usepackage{graphicx}
\usepackage[colorlinks=true, allcolors=blue]{hyperref}
\usepackage{mathrsfs}
\usepackage{tikz-cd,adjustbox}
\usepackage{spverbatim,comment}
\usepackage[shortlabels]{enumitem}

\usepackage[alphabetic,initials]{amsrefs}

\theoremstyle{plain}
\newtheorem{theorem}{Theorem}[section]
\newtheorem{lemma}[theorem]{Lemma}
\newtheorem{proposition}[theorem]{Proposition}
\newtheorem{corollary}[theorem]{Corollary}

\theoremstyle{definition}
\newtheorem{definition}[theorem]{Definition}

\theoremstyle{remark}
\newtheorem{remark}[theorem]{Remark}

\newcommand{\Rom}[1]{\uppercase\expandafter{\romannumeral#1}}

\newcommand{\CC}{\mathbb C}

\DeclareMathOperator{\Hom}{Hom}

\DeclareMathOperator{\sing}{sing}

\DeclareMathOperator{\Spec}{Spec}

\DeclareMathOperator{\codim}{codim}

\DeclareMathOperator{\ev}{ev}
\DeclareMathOperator{\Res}{Res}
\DeclareMathOperator{\tdeg}{tdeg}

\DeclarePairedDelimiter\abs{\lvert}{\rvert}

\makeatletter
\let\oldabs\abs
\def\abs{\@ifstar{\oldabs}{\oldabs*}}
\makeatother

\theoremstyle{definition} 
\newcommand{\thistheoremname}{}
\newtheorem*{genericthm}{\thistheoremname}

\newcommand{\Addresses}{{
  \vspace{\bigskipamount}
  \footnotesize
  \textsc{Department of Mathematics, Harvard University, 1 Oxford Street, Cambridge, MA 02138, USA}\par\nopagebreak
  \textit{E-mail address}: \texttt{wshen@math.harvard.edu}

  \vspace{\bigskipamount}

  \textsc{Department of Mathematics, University of Michigan, 530 Church Street, Ann Arbor, MI 48109, USA}\par\nopagebreak
  \textit{E-mail address}: 
  \texttt{srivenk@umich.edu}
  
  \vspace{\bigskipamount}
  
  \textsc{Department of Mathematics, Harvard University, 1 Oxford Street, Cambridge, MA 02138, USA}\par\nopagebreak
  \textit{E-mail address}: \texttt{ducvo@math.harvard.edu}
}}

\title{Local vanishing for toric varieties}
\author{Wanchun Shen}
\author{Sridhar Venkatesh}
\thanks{S.V. was partially supported by NSF grant DMS-2001132.}
\author{Anh Duc Vo}

\begin{document}


\maketitle

\begin{abstract}
Let $X$ be a toric variety. We establish vanishing (and non-vanishing) results for the sheaves $R^if_*\Omega^p_{\tilde X}(\log E)$, where $f: \tilde{X} \to X$ is a strong log resolution of singularities with reduced exceptional divisor $E$. These extend the local vanishing theorem for toric varieties in \cite{mustata-olano-popa}. Our consideration of these sheaves is motivated by the notion of $k$-rational singularities introduced by Friedman and Laza \cite{FL-Saito}. In particular, our results lead to criteria for toric varieties to have $k$-rational singularities, as defined in \cite{k-rational-paper}.
\end{abstract}

\section{Introduction}

Our main goal in this paper is to prove a local vanishing theorem for toric varieties. Let $X$ be an affine toric variety over $\mathbb{C}$, and $f: \tilde X \to X$ be a \textit{strong log resolution} with reduced exceptional divisor $E$, i.e. $f$ is a log resolution of singularities of $X$ which is an isomorphism over the smooth locus of $X$. When $X$ is simplicial, the vanishing of the sheaves $R^if_*\Omega^p_{\tilde X}(\log E)$ is closely related to the codimension of the singular locus of $X$. The situation is different for non-simplicial $X$, where we have non-vanishing at the first level itself. More precisely, we prove the following result.

\begin{theorem}\label{theorem:k-rational-for-toric-varieties}
     Let $X$ be an affine toric variety over $\mathbb{C}$ and $c$ be the codimension of the singular locus of $X$. Let $f: \tilde{X} \to X$ be a strong log resolution with reduced exceptional divisor $E$.
    \begin{enumerate}
    \item If $X$ is simplicial, then
    \begin{itemize}
        \item For $p<c$, we have
        \[ R^if_*\Omega^p_{\tilde X}(\log E) = 0 \quad \text{for} \quad i>0. \]
        \item For $p\geq c$, we have
        \begin{align*}   
        R^if_*\Omega^p_{\tilde X}(\log E) &= 0 \quad \text{for} \quad 0< i<c-1 \text{ or } i>p-1,\\
        R^{c-1}f_*\Omega^p_{\tilde X}(\log E) &\neq 0.
        \end{align*}
    \end{itemize}
    \item If $X $ is non-simplicial, then we have
    \begin{align*}
        R^1f_*\Omega^1_{\tilde X}(\log E) &\neq 0,\\
        R^if_*\Omega^1_{\tilde X}(\log E) &= 0 \quad \text{for} \quad i>1.
    \end{align*}
\end{enumerate}
\end{theorem}

\noindent Note that our result extends \cite[Theorem C]{mustata-olano-popa}, which states that:
\[ R^{\dim X -1}f_*\Omega^1_{\tilde X} (\log E) = 0 \] for $X$ a toric variety and  $f: \tilde{X} \to X$ a log resolution of singularities of $X$ with reduced exceptional divisor $E$ (see Remark \ref{rmk:extend-MOP}).



Another motivation for studying the sheaves $R^if_*\Omega^p_{\tilde X}(\log E)$ comes from the notion of $k$-rational singularities, introduced recently by Friedman and Laza \cite{FL-Saito}. This is a refinement of the classical notion of rational singularities and over the last few years, the development of Hodge theoretic methods has led to considerable interests in its study in the local complete intersection (lci) setting (see \cite{MOPW}, \cite{JKSY}, \cite{FL-Saito}, \cite{FL-isolated}, \cite{MP-lci}, \cite{CDM-k-rational}).
However, the current definition of $k$-rational singularities excludes most non-lci examples and need not be the right notion to consider in general, as explained in \cite{k-rational-paper}. In \textit{loc.cit.} we introduce a new definition of $k$-rational singularities and as a consequence of Theorem \ref{theorem:k-rational-for-toric-varieties}, we characterize when toric varieties have $k$-rational singularities in the new sense (\cite[Proposition E]{k-rational-paper}).



Theorem \ref{theorem:k-rational-for-toric-varieties} is actually a corollary of the following result where we restrict to the case of toric log resolutions of singularities of $X$, without assuming that the resolution is an isomorphism over the smooth locus of $X$.

\begin{theorem}\label{theorem:vanishing-in-terms-of-codim}
    Let $N$ be a free abelian group of rank $n$. Let $X$ be an affine toric variety associated to a cone $\sigma \subset N \otimes \mathbb{R}$. Let $Y$ be the toric variety associated to a fan $\Sigma$ which refines $\sigma$, such that the induced toric morphism $\pi:Y \to X$ is a log resolution of singularities with reduced exceptional divisor $E$. Let $Z$ be the complement of the domain of $\pi^{-1}$ and let $c:=\codim_X Z$.
\begin{enumerate}
    \item If $X$ is simplicial, then
    \begin{enumerate}
        \item For $p<c$, we have
        \[ R^i\pi_*\Omega^p_{Y}(\log E)  = 0 \quad \text{for} \quad i>0. \]
        \item For $p\geq c$, we have
        \begin{align*}   
        R^i\pi_*\Omega^p_{Y}(\log E) &= 0 \quad \text{for} \quad 0< i<c-1 \text{ or } i>p-1,\\
        R^{c-1}\pi_*\Omega^p_{Y}(\log E) &\neq 0.
        \end{align*}
    \end{enumerate}
    \item If $X $ is non-simplicial, then we have
    \begin{align*}
        R^1\pi_*\Omega^1_{Y}(\log E) &\neq 0,\\
        R^i\pi_*\Omega^1_{Y}(\log E) &= 0 \quad \text{for} \quad i>1.        
    \end{align*}
\end{enumerate}
\end{theorem}

\noindent Theorem \ref{theorem:k-rational-for-toric-varieties} can be easily deduced from Theorem \ref{theorem:vanishing-in-terms-of-codim} by taking the toric morphism $\pi$ to be a strong log resolution of singularities with reduced exceptional divisor $E$ (such a $\pi$ exists by \cite[Theorem 11.2.2]{cox-little-schenk:toric-varieties}), and then using \cite[Lemma 1.6]{MP-lci} which says that the sheaves $R^if_*\Omega^p_{\tilde X}(\log E)$ only depend on $X$ if $f$ is a strong log resolution.

The proof of Theorem \ref{theorem:vanishing-in-terms-of-codim} relies on the existence of a $\pi_*$-acyclic resolution of $\Omega^p_{Y}(\log E)$, which will help us compute $R^i\pi_*\Omega^p_Y(\log E)$. This resolution is analogous to the Ishida complex in \cite[Lemma 3.5]{Oda} (see also \cite{Ishida}), which is a resolution of $\Omega^p_{Y}$. While we can prove the existence of such a resolution of $\Omega^p_{Y}(\log E)$ by suitably modifying the proof of \cite[Theorem 3.6]{Oda}, we instead give an alternate self-contained proof using Koszul complexes.

\textbf{Outline of the paper.} We recall some basic facts about toric varieties and about log differentials in Section \ref{section:preliminaries}. In Section \ref{section:proof-of-(2)}, we prove (2) in Theorem \ref{theorem:vanishing-in-terms-of-codim} since it contains the main idea of the proof in an easier setting. Next, we study certain quotient complexes of a Koszul complex in Section \ref{section:koszul-complexes}, which will help us obtain a $\pi_*$-acyclic resolution of $\Omega^p_{Y}(\log E)$ in Section \ref{section:complexes-on-toric-varieties}. Finally, we finish the proof of Theorem \ref{theorem:vanishing-in-terms-of-codim} in Section \ref{section:proof-of-(1)}.

\textbf{Acknowledgements.} We would like to express our sincere gratitude to Mircea Musta\c{t}\u{a} and Mihnea Popa for their constant support during the preparation of this paper.


\section{Preliminaries}\label{section:preliminaries}

Throughout the paper, by a variety we mean an integral scheme of finite type over $\CC$. For a variety $X$, we denote its singular locus by $X_{\sing}$.

\subsection{Remarks about toric varieties}

Fix a free abelian group $N$ of rank $n$ and let $M := \Hom_{\mathbb{Z}}(N,\mathbb{Z})$. To a strongly convex rational polyhedral cone $\sigma \subset N \otimes \mathbb{R}$, we associate an $n$-dimensional affine toric variety $X_\sigma$. More generally, to a fan $\Delta \subset N \otimes \mathbb{R}$, we associate an $n$-dimensional toric variety $X_\Delta$. For general notions regarding toric varieties, we refer to \cite{fulton:toric-varieties} and \cite{cox-little-schenk:toric-varieties}.


We gather some basic facts about toric varieties in the next remark.
\begin{remark}\label{remark:general-toric-varieties}
    Let $X_\Delta$ be a toric variety with its associated fan $\Delta \subset N \otimes \mathbb{R}$.
    \begin{enumerate}[label=(\alph*)]
        \item $X_\Delta$ is normal (\cite[Theorem 1.3.5]{cox-little-schenk:toric-varieties}). In fact, $X_\Delta$ has rational singularities (\cite[Theorem 11.4.2]{cox-little-schenk:toric-varieties}).
        \item Given any cone $\tau \subset \Delta$, we get an inclusion of the associated affine toric variety $X_\tau \hookrightarrow X_\Delta$ as an open subset. We denote this affine open subset by $U_\tau \subset X_\Delta$. 
        \item Following \cite[Section 3.1]{fulton:toric-varieties}, given any $r$ dimensional cone $\tau \subset \Delta$, we get a corresponding torus invariant subvariety $V(\tau) \subset X_\Delta$ of codimension $r$. The lattice and dual lattice of $V(\tau)$ are given by:
    \[ N(\tau):= \frac{N}{N \cap \tau}, \quad M(\tau):= M \cap \tau^\perp.  \]
    \item If $\tau$ is a $1$-dimensional cone, and $v \in N \cap \tau$ is a generator of $N \cap \tau$, then we have an isomorphism:
    \begin{equation}\label{equation:evaluation-residue-map}
    \begin{split}
        \ev_v : \frac{M}{M \cap \tau^\perp} &\xrightarrow{\sim}\mathbb{Z}\\
        \overline{f} &\mapsto f(v).
    \end{split}
    \end{equation}
    \item Given a cone $\sigma \subset N \otimes \mathbb{R}$, the affine toric variety $X_\sigma = \Spec \mathbb{C}[\sigma^\vee \cap M]$ carries an action of the torus $T = \Spec \mathbb{C}[M]$. The ring $\mathbb{C}[\sigma^\vee \cap M]$ decomposes into $T$-eigenspaces
    \[ \mathbb{C}[\sigma^\vee \cap M] = \bigoplus_{u \in \sigma^\vee \cap M} \mathbb{C} \cdot \chi^u, \]
    and so it carries an $M$-grading. Similarly, for any torus invariant subvariety $V(\tau) = \Spec \mathbb{C}[\sigma^\vee \cap M \cap \tau^\perp]$ of $X_{\sigma}$, the ring $\mathbb{C}[\sigma^\vee \cap M \cap \tau^\perp]$ decomposes into $T$-eigenspaces 
    \[ \mathbb{C}[\sigma^\vee \cap M \cap \tau^\perp] = \bigoplus_{u \in \sigma^\vee \cap M \cap \tau^\perp} \mathbb{C} \cdot \chi^u, \]
    and so it carries an $M$-grading as well.
    \end{enumerate}
\end{remark}


In the next remark, we make some observations about simplicial toric varieties and their resolutions.
\begin{remark}\label{remark:intersection-of-strict-transforms-toric-varieties}
    Let $X:=X_\sigma$ be an affine toric variety such that its associated cone $\sigma \subset N \otimes \mathbb{R}$ is simplicial. Let $\tau_1,\dots,\tau_m$ be a subset of the rays of $\sigma$ and let $F_1,\dots,F_m$ be the associated torus invariant divisors on $X$. Let $\tau$ denote the face of $\sigma$ spanned by $\tau_1,\dots,\tau_m$. Let $U_\tau$ denote the associated affine open subset of $X$.

\noindent    Since $\sigma$ is simplicial, observe that the torus invariant subvariety $V(\tau)$ of $X$ corresponding to the face $\tau$ is just $F_1 \cap \dots \cap F_m$, the scheme theoretic intersection of the $F_i$'s. This is because we have $\tau^\perp = \tau_1^\perp \cap \dots \cap \tau_m^\perp$, hence:
    \begin{align*}
        V(\tau) &= \Spec \mathbb{C}[\sigma^\vee \cap M \cap \tau^\perp]\\
        &= \Spec \frac{\mathbb{C}[\sigma^\vee \cap M]}{(\chi^u \mid u \in \sigma^\vee \cap M, u \notin \tau^\perp)}\\
        &= \Spec \frac{\mathbb{C}[\sigma^\vee \cap M]}{(\chi^u \mid u \in \sigma^\vee \cap M, u \notin \tau_1^\perp) + \dots + (\chi^u \mid u \in \sigma^\vee \cap M, u \notin \tau_m^\perp)}\\
        &= \Spec \frac{\mathbb{C}[\sigma^\vee \cap M]}{(\chi^u \mid u \in \sigma^\vee \cap M, u \notin \tau_1^\perp)} \cap \dots \cap \Spec \frac{\mathbb{C}[\sigma^\vee \cap M]}{(\chi^u \mid u \in \sigma^\vee \cap M, u \notin \tau_m^\perp)}\\
        &= F_1 \cap \dots \cap F_m
    \end{align*}
    where we use the notation as in \cite[Section 1.3]{fulton:toric-varieties} and \cite[Section 3.1]{fulton:toric-varieties}.

    Now, let $\Sigma$ be a smooth fan refining $\sigma$ and let $Y:=X_\Sigma$ be the associated smooth toric variety. Let $\pi:Y \to X$ denote the associated projective birational morphism. Since $\Sigma$ is a refinement of $\sigma$, we get that $\tau_1,\dots,\tau_m$ are $1$-dimensional cones in $\Sigma$. Let $D_1,\dots,D_m$ be the corresponding torus invariant divisors in $Y$. Observe that $D_1,\dots,D_m$ are the strict transforms of $F_1,\dots,F_m$.
    
\noindent    We have that either $\tau$ is a cone in $\Sigma$ or it is not.
    \begin{enumerate}
        \item If $\tau$ is a cone in $\Sigma$, this means that $U_\tau$ is also an open subset of $Y$ and that $\pi|_{U_\tau}$ is an isomorphism. Since $D_{1} \cap \dots \cap D_m \cap U_\tau \neq \emptyset$, we have that $\pi :D_{1} \cap \dots \cap D_m \to F_{1} \cap \dots \cap F_m$ is a projective birational morphism from the smooth variety $D_{1} \cap \dots \cap D_m$. In particular, this implies that
        \begin{align*}
            \pi_*\mathcal{O}_{D_{1} \cap \dots \cap D_m} = \mathcal{O}_{F_{1} \cap \dots \cap F_m}, \quad R^i\pi_*\mathcal{O}_{D_{1} \cap \dots \cap D_m} = 0 \text{ for all $i>0$}
        \end{align*}
        since $F_{1} \cap \dots \cap F_m$ is a toric variety, so it is normal and has rational singularities (see Remark \ref{remark:general-toric-varieties}(a)).
        \item If $\tau$ is not a cone in $\Sigma$, then there is no cone in $\Sigma$ containing $\tau_{1},\dots,\tau_m$. Hence, $D_{1} \cap \dots \cap D_m = \emptyset$.
    \end{enumerate}
    To summarize, either $\pi(D_{j_1} \cap \dots \cap D_{j_m}) = F_{j_1} \cap \dots \cap F_{j_m}$ or $D_{j_1} \cap \dots \cap D_{j_m} = \emptyset$.
\end{remark} 

\subsection{Log differentials and residues }
In this subsection, we discuss a particular sequence of log differentials on a smooth variety. This sequence will be crucial to our proof later.
\begin{remark}\label{remark:residue-sequence-smooth-varieties}
    Let $Y$ be a smooth variety and let $D= D_1 + \dots+ D_k + D_{k+1} + \dots +  D_r$ be a simple normal crossing divisor on $Y$.  Let $E:= D_{k+1} + \dots + D_r$. For any $m \leq r$, define:
    \[ T_{j_1,\dots,j_m} := (D - D_{j_1} - \dots - D_{j_m})|_{D_{j_1}\cap \dots \cap D_{j_m}}. \]
    This is a simple normal crossing divisor on $D_{j_1}\cap \dots \cap D_{j_m}$. For any $p \leq n$ and for any $j$, we have the natural residue map:
    \begin{align*}
        \Res_p: \Omega^p_Y(\log D) &\to \Omega^{p-1}_{D_j} (\log T_j).
    \end{align*}
    Observe that for any $j' \neq j$, we also have the residue map:
    \begin{align*}
        \Res_{p-1}: \Omega^{p-1}_{D_j} (\log T_j) \to \Omega^{p-2}_{D_j \cap D_{j'}} (\log T_{j,j'}).
    \end{align*}
    We continue this sequence, while taking all the $D_j$'s together for $j=1,\dots,k$, to get:
    \begin{equation}\label{equation:smooth-varieties-residue-sequence}
     \begin{split}
         \widetilde{C_p} : \quad 0 \to \Omega^p_Y(\log D) \xrightarrow{\Res_p} \bigoplus_{j=1}^k \Omega^{p-1}_{D_j} (\log T_j) \xrightarrow{\Res_{p-1}} \bigoplus_{1 \leq j_1<j_2 \leq k} \Omega^{p-2}_{D_{j_1} \cap D_{j_2}}(\log T_{j_1,j_2}) \to \dots\\
        \dots \to \bigoplus_{1 \leq j_1 < \dots < j_p \leq k} \mathcal{O}_{D_{j_1}\cap \dots \cap D_{j_p}} \to 0 
     \end{split}
    \end{equation}
    where $\Omega^p_Y(\log D)$ is in cohomological degree $0$ (we follow the convention that if $D_{j_1}\cap \dots \cap D_{j_m} = \emptyset$, then we set $\Omega^{p-m}_{D_{j_1}\cap \dots \cap D_{j_m}}(\log T_{j_1,\dots,j_m}) = 0$). We will prove later in Corollary \ref{corollary:resolutioin-for-OmegapY(log E)} that $\widetilde{C_p}$ is a resolution of $\Omega^p_Y(\log E)$.
\end{remark}

We now discuss log differentials on smooth toric varieties.
\begin{remark}\label{remark:trivialization-of-log-differentials-for-toric-varieties}
    Let $Y:=X_\Delta$ be a smooth toric variety with its associated fan $\Delta \subset N \otimes \mathbb{R}$. Let $\tau_1,\dots,\tau_r$ be all the 1-dimensional cones in $\Delta$ and let $D_1,\dots, D_r$ be the corresponding torus invariant divisors on $Y$. Denote $D = D_1 + \dots + D_r$.
    \begin{enumerate}
        \item By \cite[Section 4.3]{fulton:toric-varieties}, we have the isomorphism:
\begin{align*}
    M \otimes \mathcal{O}_Y \cong \Omega^1_Y(\log D).
\end{align*}
Observe that this isomorphism behaves well with respect to the residue map i.e. we have the following commuting diagram for any $i$:
\[
\begin{tikzcd}
    M \otimes \mathcal{O}_Y \ar[r] \ar[d, swap, "\varphi \otimes (\cdot)|_{D_i}"] & \Omega^1_Y(\log D) \ar[d, "\Res_i"]\\
    \frac{M}{M \cap \tau_i^\perp} \otimes \mathcal{O}_{D_i} \ar[r, "\ev_i"] & \mathcal{O}_{D_i}
\end{tikzcd}
\]
where $\varphi$ is the natural surjection $M \to \frac{M}{M \cap \tau_i^\perp}$, $(\cdot)|_{D_i}$ is the restriction to $D_i$, $\Res_i$ corresponds to taking the residue along the divisor $D_i$ and $\ev_i$ comes from (\ref{equation:evaluation-residue-map}).
        \item Observe that the divisor $(D-D_i)|_{D_i}$ on the toric variety $D_i$ is the sum of the torus invariant divisors of $D_i$. Thus, by \cite[Section 4.3]{fulton:toric-varieties}, we have:
\[ (M \cap \tau_i^\perp) \otimes \mathcal{O}_{D_i} \cong \Omega^1_{D_i}(\log (D-D_i)|_{D_i}). \]
Taking exterior powers, we have:
\begin{align*}
    \bigwedge^p M \otimes \mathcal{O}_Y &\cong \Omega^p_Y(\log D).\\
    \bigwedge^{p-1}(M \cap \tau_i^\perp) \otimes \mathcal{O}_{D_i} &\cong \Omega^{p-1}_{D_i}(\log (D-D_i)|_{D_i}).
\end{align*}
As before, observe that the following diagram commutes:
\[
\begin{tikzcd}
    \bigwedge^p M \otimes \mathcal{O}_Y \ar[r] \ar[d, swap, "\varphi \otimes (\cdot)|_{D_i}"] & \Omega^p_Y(\log D) \ar[d, "\Res_i"]\\
    \bigwedge^{p-1}(M \cap \tau_i^\perp) \otimes \frac{M}{M \cap \tau_i^\perp} \otimes \mathcal{O}_{D_i} \ar[r, "\ev_i"] & \Omega^{p-1}_{D_i}(\log (D-D_i)|_{D_i}),
\end{tikzcd}
\]
where $\varphi$ is the natural surjection $\bigwedge^p M \to \bigwedge^{p-1}(M \cap \tau_i^\perp)$, $(\cdot)|_{D_i}$ is the restriction to $D_i$, $\Res_i$ corresponds to taking the residue along the divisor $D_i$ and $\ev_i$ comes from (\ref{equation:evaluation-residue-map}).
    \end{enumerate}
\end{remark}

We end the section by discussing the result of \cite{mustata-olano-popa} that was stated in the introduction.

\begin{remark}\label{rmk:extend-MOP}
Note that Theorem \ref{theorem:k-rational-for-toric-varieties} says that if $X$ is a toric variety and $f: \tilde X \to X$ is a strong log resolution of singularities with reduced exceptional divisor $E$, then:
\begin{enumerate}
\item $R^i f_*\Omega^1_{\tilde{X}}(\log E)=0$ for $i\ge 1$, if $X$ is simplicial;
    \item $R^i f_*\Omega^1_{\tilde{X}}(\log E)=0$ for $i\ge 2$, if $X$ is non-simplicial. 
\end{enumerate}
Since all toric varieties of dimension $2$ are simplicial, we see that
\[R^{\dim X-1} f_*\Omega^1_{\tilde{X}}(\log E)=0.\]
This implies \cite[Theorem C]{mustata-olano-popa} since it suffices to prove the statement for any log resolution of singularities with reduced exceptional divisor $E$ (\cite[Lemma 1.1]{mustata-olano-popa}).
\end{remark}

\section{Proof of (2) in Theorem \ref{theorem:vanishing-in-terms-of-codim}}\label{section:proof-of-(2)}
Let us first prove (2) in Theorem \ref{theorem:vanishing-in-terms-of-codim} since it contains the main idea of the proof in an easier setting.

\noindent We have the following lemma whose proof essentially follows from \cite[Section 4.3]{fulton:toric-varieties}.
\begin{lemma}\label{lemma:ses-for-Omega1Y(log E)}
    Let $Y$ be a smooth variety. Let $D = D_1+\dots+D_k + D_{k+1} + \dots + D_r$ be a simple normal crossing divisor on $Y$. Denote $E := D_{k+1} + \dots + D_r$. Then, we have a short exact sequence:
    \[ 0 \to \Omega^1_Y(\log E) \to \Omega^1_Y(\log D) \xrightarrow{\varphi} \bigoplus_{i=1}^k \mathcal{O}_{D_i} \to 0. \]
\end{lemma}

We will also need the following lemma to prove the required non-vanishing in (2) of Theorem \ref{theorem:vanishing-in-terms-of-codim}.
\begin{lemma}\label{lemma:non-simplicial-first-residue-map}
     Let $\sigma \subset N \otimes \mathbb{R}$ be a non-simplicial cone, with $\tau_1,\dots,\tau_k$ denoting all its rays. Let $X:= X_\sigma$ be the corresponding affine toric variety, with $F_1,\dots,F_k$ being the torus invariant divisors associated to $\tau_1,\dots,\tau_k$. Then, the natural map
    \[ \phi: M \otimes \mathcal{O}_X \to \bigoplus_{j=1}^k (\frac{M}{M \cap \tau_j^\perp} \otimes \mathcal{O}_{F_j}) \]
    is \textbf{not} surjective.
\end{lemma}
\begin{proof}
    Let $L \subset N \otimes \mathbb{R}$ be the subspace spanned by $\tau_1,\dots,\tau_k$, and let its dimension be $l$. Observe that $l < k$ since $\sigma$ is non-simplicial. The above map $\phi$ factors as follows:
    \[ M \otimes \mathcal{O}_X \to \frac{M}{M \cap L^\perp} \otimes \mathcal{O}_X \xrightarrow{\psi} \bigoplus_{j=1}^k (\frac{M}{M \cap \tau_j^\perp} \otimes \mathcal{O}_{F_j}). \]
    Hence it suffices to show that $\psi$ is not surjective, which is what we will do. Observe that $\sigma$ corresponds to a torus invariant subset $V(\sigma) \subset X$. Since $\tau_j \subset \sigma$ for all $j = 1,\dots,k$, we have $V(\sigma) \subset V(\tau_j) = F_j$ for all $j = 1,\dots,k$. Now, if we pick a point $x \in V(\sigma)$, this lies in every $F_j$. Thus, tensoring $\psi$ by the residue field $k(x)$, we get:
    \[ \overline{\psi} : \frac{M}{M \cap L^\perp} \otimes k(x) \to  \bigoplus_{j=1}^k (\frac{M}{M \cap \tau_j^\perp} \otimes k(x)).  \]
    Now observe that we have
    \[\dim_{k(x)} (\frac{M}{M \cap L^\perp} \otimes k(x)) = l < k = \dim_{k(x)} \left(\bigoplus_{j=1}^k (\frac{M}{M \cap \tau_j^\perp} \otimes k(x)) \right). \]
    Thus, $\overline{\psi}$ is not surjective, hence $\psi$ is not surjective.
\end{proof}

\begin{proof}[Proof of (2) in Theorem \ref{theorem:vanishing-in-terms-of-codim}]
Let $\tau_1,\dots,\tau_k$ be the rays of $\sigma$ and let $F_1,\dots,F_k$ be the corresponding torus invariant divisors in $X$. Since $\Sigma$ is a refinement of $\sigma$, the $1$-dimensional cones of $\Sigma$ are $\tau_1,\dots,\tau_k$ along with some additional rays $\tau_{k+1},\dots \tau_r$. Let $D_1,\dots,D_r$ be the corresponding torus invariant divisors in $Y$. Observe that $D_1,\dots,D_k$ are the strict transforms of $F_1,\dots,F_k$. Define the following (simple normal crossing) divisors on $Y$: $E =  D_{k+1} + \dots + D_r$ and $D = D_1 + \dots + D_k + E$.

By Lemma \ref{lemma:ses-for-Omega1Y(log E)}, we have an exact sequence
    \[ 0 \to \Omega^1_Y(\log E) \to \Omega^1_Y(\log D) \xrightarrow{\varphi} \bigoplus_{j=1}^k \mathcal{O}_{D_j} \to 0. \]
     By Remark \ref{remark:trivialization-of-log-differentials-for-toric-varieties}, we have $\Omega_Y^1(\log D) \simeq M \otimes \mathcal{O}_Y$. Under this identification, the above sequence becomes
    \begin{align*}
      0 \to \Omega^1_Y(\log E) \to  M \otimes \mathcal{O}_Y \xrightarrow{\phi} \bigoplus_{j=1}^k \frac{M}{M \cap \tau_j^\perp} \otimes \mathcal{O}_{D_j} \to 0.
    \end{align*}
    For each $j$, observe that $\pi|_{D_j}:D_j \to F_j$ is a projective birational morphism from the smooth variety $D_j$ to the toric variety $F_j$. Since toric varieties are normal (see Remark \ref{remark:general-toric-varieties}), we get $\pi_*\mathcal{O}_{D_j} = \mathcal{O}_{F_j}$. Moreover, since $X$ has rational singularities (see Remark \ref{remark:general-toric-varieties}), we have: $R^i\pi_*\mathcal{O}_{Y} = 0$ for all $i>0$. Thus, when we apply $\pi_*$ to the above sequence, we get the following exact sequence:
    \[ 0 \to \pi_*\Omega^1_Y(\log E) \to M \otimes \mathcal{O}_X \xrightarrow{\phi} \bigoplus_{j=1}^k \frac{M}{M \cap \tau_j^\perp} \otimes \mathcal{O}_{F_j} \to R^1\pi_*\Omega^1_Y(\log E) \to 0 \]
    and that $R^i\pi_*\Omega^1_{Y}(\log E) = 0 \text{ for $i>1$}$. Finally, Lemma \ref{lemma:non-simplicial-first-residue-map} implies that $\phi$ is not surjective, and so, $R^1\pi_*\Omega^1_Y(\log E) \neq 0$ as required.
\end{proof}

\section{Quotient complexes of the Koszul complex}\label{section:koszul-complexes}

In this section and the next, we generalize Lemma \ref{lemma:ses-for-Omega1Y(log E)} to obtain a $\pi_*$-acyclic resolution of $\Omega^p_{Y}(\log E)$.

For every $n \in \mathbb{N}$, consider the polynomial ring $R = \mathbb{C}[x_1,\dots,x_n]$, with the natural multi-grading given by setting $\deg(x_i) = \delta_i := (0,\dots,0,1,0,\dots,0)$ with $1$ at the $i$-th slot. Let $V := \mathbb{C}^n$ be the standard $n$-dimensional vector space over $\mathbb{C}$ with basis $\{e_1,\dots,e_n\}$. Consider the trivial $R$-module associated to $V$:
\[ E := R \cdot e_1 \oplus \dots \oplus R \cdot e_n. \]
Define a multi-grading on $E$ by setting $\deg(e_i) =\delta_i$. Consider the total grading induced by the multi-grading: if $\deg(f) = (a_1,\dots,a_n)$, then define the total degree to be $\tdeg(f) = \sum_i a_i$. Consider the following $R$-module map:
    \begin{align*}
        s: E &\to R\\
        e_i &\mapsto x_i.
    \end{align*}
    Observe that $s$ respects the multi-grading. Since the $x_i$ form a regular sequence, the Koszul complex associated to $s$ is exact:
    \begin{align*}
        0 \to \bigwedge^n E \xrightarrow{d_n} \bigwedge^{n-1} E \xrightarrow{d_{n-1}} \dots \to \bigwedge^1 E \xrightarrow{d_1} R \to R/(x_1,\dots,x_n) \to 0,
    \end{align*}
    with the differential $d_r$ given as:
    \[ d_r(f_1 \wedge \dots \wedge f_r) := \sum_{i=1}^r (-1)^{i+1} s(f_i) f_1 \wedge \dots \wedge \widehat{f_i} \wedge \dots \wedge f_r. \]
    The differentials respect the multi-grading as well. Now, consider the exact sequence in a fixed total degree $p>0$:
    \begin{align*}
        0 \to (\bigwedge^p E)_p \xrightarrow{d_p} (\bigwedge^{p-1} E)_p \xrightarrow{d_{p-1}} \dots (\bigwedge^1 E)_p \xrightarrow{d_1} R_p \to 0.
    \end{align*}
    (Observe that $(R/(x_1,\dots,x_n))_p =0$ and $(\bigwedge^k E)_p = 0$ for all $k>p$.) We can write the terms of this complex as a direct sum as follows:
    \[ (\bigwedge^{p-m} E)_p = A_{p-m} \oplus B_{p-m} \]
    where $A_{p-m}$ is the $\mathbb{C}$-span of all $\omega = x_{j_1}\dots x_{j_m} \cdot e_{i_1} \wedge \dots \wedge e_{i_{p-m}}$ such that if $\deg(\omega) =(a_1,\dots,a_n)$, then all $a_i \leq 1$ (this is the same as saying that $i_r \neq j_s \text{ for any } r,s$, and $j_s \neq j_{s'}$ for any $s,s'$), and $B_{p-m}$ is the $\mathbb{C}$-span of all $\omega = x_{j_1}\dots x_{j_m} \cdot e_{i_1} \wedge \dots \wedge e_{i_{p-m}}$ such that if $\deg(\omega) =(a_1,\dots,a_n)$, then some $a_i >1$. Since the differentials respect the multi-grading, we deduce that $d_{p-m}(A_{p-m}) \subset A_{p-m-1}$ and $d_{p-m}(B_{p-m}) \subset B_{p-m-1}$ for all $m$. Thus we get an exact sequence:
    \begin{align*}
        0 \to A_p \to A_{p-1} \to \dots \to A_1 \to R_p \to 0.
    \end{align*}

Now, let $L_j \subset V$ be the subspace spanned by $\{e_1,\dots,\widehat{e_j},\dots,e_n\}$. Thus $\frac{V}{L_j}$ is spanned by $\overline{e_j}$, the image of $e_j$. Denote $L_{j_1,\dots,j_m} := L_{j_1} \cap \dots \cap L_{j_m}$ and $V_{j_1,\dots,j_m} := \frac{V}{L_{j_1}} \otimes \dots \otimes \frac{V}{L_{j_m}}$. With these identifications, observe that:
    \begin{align*}
        A_{p-m} = \bigoplus_{j_1<\dots<j_m} \bigwedge^{p-m} L_{j_1,\dots,j_m} \otimes V_{j_1,\dots,j_m}
    \end{align*}
where we identify $x_{j_1}\dots x_{j_m}$ with the basis element $\overline{e_{j_1}}\otimes \dots \otimes \overline{e_{j_m}}$ of $V_{j_1,\dots,j_m}$. Thus we have an exact sequence:
\begin{equation}\label{equation:koszul-subcomplex-n}
\resizebox{\textwidth}{!}{$
    0 \to \bigwedge^p V \to \displaystyle\bigoplus_{j=1}^n \left( \bigwedge^{p-1} L_j \otimes \frac{V}{L_j} \right) \to \bigoplus_{1\leq j_1<j_2 \leq n} \left( \bigwedge^{p-2} L_{j_1,j_2} \otimes V_{j_1,j_2}\right) \to \dots \bigoplus_{1 \leq j_1 < \dots < j_p \leq n} V_{j_1,\dots,j_p} \to 0,
$}
\end{equation}
with $\displaystyle\bigwedge^p V$ living in cohomological degree $-p$. For each $k \leq n$, observe that we have the following subcomplex:
\begin{equation}
    \begin{split}
        0 \to 0 \to \bigoplus_{j=k+1}^{n} \left( \displaystyle \bigwedge^{p-1} L_j \otimes \frac{V}{L_j} \right) \to \bigoplus_{\substack{1\leq j_1<j_2 \leq n,\\ k+1\leq j_2}} \left( \bigwedge^{p-2} L_{j_1,j_2} \otimes V_{j_1,j_2}\right)& \to \dots \\
        \dots \to \bigoplus_{\substack{1 \leq j_1 < \dots < j_p \leq n,\\ k+1 \leq j_p}}& V_{j_1,\dots,j_p} \to 0.
    \end{split}
\end{equation}
Consider the corresponding quotient complex, placed in cohomological degrees $-p,\dots,0$:
\begin{equation}\label{equation:koszul-subcomplex-k}
\begin{split}
    C_{n,k,p}: \quad 0 \to \bigwedge^p V \to \bigoplus_{j=1}^k \left( \bigwedge^{p-1} L_j \otimes \frac{V}{L_j} \right) \to \bigoplus_{1\leq j_1<j_2 \leq k} \left( \bigwedge^{p-2} L_{j_1,j_2} \otimes V_{j_1,j_2} \right) \to \dots\\
    \dots \to \bigoplus_{1 \leq j_1 < \dots < j_p \leq k} V_{j_1,\dots,j_p} \to 0.
\end{split}
\end{equation}

\begin{lemma}\label{lemma:main-linear-algebra-result}
    For every $n \in \mathbb{N}$, for every $0 \leq k \leq n$ and for every $0<p \leq n$, the complex $C_{n,k,p}$ is exact at all cohomological degrees other than $-p$.
\end{lemma}
\begin{proof}
    Observe that we have already proved the lemma for $C_{n,n,p}$ in (\ref{equation:koszul-subcomplex-n}). To prove the lemma in general, the idea is to use (ascending) induction on $p$, and (descending) induction on $k$.
    
    We first induct on $p$. For the base case, observe that when $p=1$, then for any $n$ and for any $k \leq n$, we get:
    \[C_{n,k,1}:  \quad 0 \to  V\xrightarrow{\phi} \displaystyle \bigoplus_{j=1}^{k} \frac{V}{L_{j}} \to 0.\]
    This is exact in cohomological degree $0$ since $\phi$ is clearly surjective. Now, fix $p_0$, and assume that we have proved the lemma for $p \leq p_0-1$, and for any $n$ and for any $k \leq n$.

    Now fixing an $n$ and $k_0$, let us attempt to prove the lemma for $C_{n,k_0,p_0}$. Observe that by (\ref{equation:koszul-subcomplex-n}), we have proved the statement for $C_{n,n,p_0}$. Thus by descending induction on $k$, assume that we have proved the statement for $C_{n,k,p_0}$ for all $k_0+1\leq k \leq n$.
    
    Consider the following short exact sequence of complexes:
    \[\begin{tikzcd}
        & 0 \ar[d] & 0 \ar[d]\\
        0 \ar[r] \ar[d] & \displaystyle \bigwedge^{p_0-1} \left( L_{k_0+1} \otimes \frac{V}{L_{k_0+1}} \right) \ar[r] \ar[d] & \displaystyle \bigoplus_{j=1}^{k_0} \left( \displaystyle \bigwedge^{p_0-2} (L_{j} \cap L_{k_0+1}) \otimes \frac{V}{L_{j}} \otimes \frac{V}{L_{k_0+1}}\right) \ar[d] \dots\\
        \displaystyle \bigwedge^{p_0} V \ar[r] \ar[d, equal] & \displaystyle \bigoplus_{j=1}^{k_0+1} \left( \displaystyle \bigwedge^{p_0-1} L_j \otimes \frac{V}{L_j} \right) \ar[r] \ar[d] & \displaystyle \bigoplus_{1\leq j_1<j_2 \leq k_0+1} \left( \displaystyle \bigwedge^{p_0-2} (L_{j_1} \cap L_{j_2}) \otimes \frac{V}{L_{j_1}} \otimes \frac{V}{L_{j_2}}\right) \ar[d] \dots\\
        \displaystyle \bigwedge^{p_0} V \ar[r] & \displaystyle \bigoplus_{j=1}^{k_0} \left( \displaystyle \bigwedge^{p_0-1} L_j \otimes \frac{V}{L_j} \right) \ar[d] \ar[r] & \displaystyle \bigoplus_{1\leq j_1<j_2 \leq k_0} \left( \displaystyle \bigwedge^{p_0-2} (L_{j_1} \cap L_{j_2}) \otimes \frac{V}{L_{j_1}} \otimes \frac{V}{L_{j_2}}\right) \ar[d] \dots \\
        & 0 & 0.
    \end{tikzcd}\]
    By the induction hypothesis, the central row, which is just $C_{n,k_0+1,p_0}$, is exact at cohomological degrees other than $-p_0$. The top row is just $C_{n-1,k_0,p_0-1} \otimes \frac{V}{L_{k_0+1}}$, hence it is exact at cohomological degrees other than $-(p_0-1)$ by the induction hypothesis. Thus, the bottom row is exact at cohomological degrees other than $-p_0$, finishing the proof.
\end{proof}

\section{Complexes on toric varieties}\label{section:complexes-on-toric-varieties}\label{section:complexes-on-toric-varieties}

As a consequence of Lemma \ref{lemma:main-linear-algebra-result}, we obtain in this section the required $\pi_*$-acyclic resolution of $\Omega^p_{Y}(\log E)$ in Corollary \ref{corollary:resolutioin-for-OmegapY(log E)}. We also partially recover \cite[Theorem 3.6.3]{Oda}, which will be useful for our proof later. 

Consider the following additional notation: let $\sigma \subset N \otimes \mathbb{R}$ be a simplicial cone and let $X:=X_\sigma$ be the corresponding affine toric variety. Let $\tau_1,\dots,\tau_k$ denote a subset of the rays in $\sigma$ and let $F_1,\dots,F_k$ be the corresponding torus invariant divisors of $X$. Denote $F_{j_1,\dots,j_m}:= F_{j_1} \cap \dots \cap F_{j_m}$.

Consider the following sequence:
\begin{equation}\label{equation:toric-residue-sequence}
\resizebox{\textwidth}{!}{$
    \begin{split}
        \bigwedge^p V \otimes \mathcal{O}_X \xrightarrow{\widehat{d_p}} \bigoplus_{j=1}^k \left( \bigwedge^{p-1} L_j \otimes \frac{V}{L_j} \otimes \mathcal{O}_{F_j} \right) \xrightarrow{\widehat{d_{p-1}}} \bigoplus_{1 \leq j_1<j_2 \leq k} \left( \bigwedge^{p-2} (L_{j_1,j_2}) \otimes V_{j_1,j_2} \otimes \mathcal{O}_{F_{j_1,j_2}} \right) \to \dots\\
        \dots \bigoplus_{1 \leq j_1 < \dots < j_p \leq k} \left( V_{j_1,\dots,j_p} \otimes \mathcal{O}_{F_{j_1,\dots,j_p}} \right)\to 0.
    \end{split}
$}
\end{equation}
where $\displaystyle \bigwedge^p V \otimes \mathcal{O}_X$ lies in cohomological degree $0$. We define the differentials $\widehat{d_r}$ as before while also factoring in a restriction map:
\begin{align*}
    \widehat{d_r}(e_{i_1} &\wedge \dots \wedge e_{i_r} \otimes \overline{e_{j_1}} \otimes \dots \otimes \overline{e_{j_m}} \otimes f) :=\\
    &\sum_{l=1}^r (-1)^{l+1} e_{i_1} \wedge \dots \wedge \widehat{e_{i_l}} \wedge \dots\wedge e_{i_r} \otimes \overline{e_{i_l}} \otimes \overline{e_{j_1}} \otimes \dots \otimes \overline{e_{j_m}} \otimes f|_{F_{i_l,j_1,\dots,j_m}},
\end{align*}
where $f \in \mathcal{O}_{F_{j_1,\dots,j_m}}$.

\begin{corollary}\label{corollary:toric-residue-sequence}
The sequence (\ref{equation:toric-residue-sequence}) is an exact complex at all cohomological degrees other than $0$.
\end{corollary}
\begin{proof}
    Observe that by Remark \ref{remark:general-toric-varieties}.(e), every $\mathcal{O}_{F_{j_1,\dots,j_m}} = \displaystyle\bigoplus_{u \in \sigma^\vee \cap M \cap \tau_{j_1}^\perp \cap \dots \cap \tau_{j_m}^\perp} \mathbb{C} \cdot \chi^u$ is an $M$-graded ring.
    Thus, the terms of the above sequence (\ref{equation:toric-residue-sequence}) all inherit an $M$-grading. Since the map given by restriction of functions $\mathcal{O}_{F_{j_1,\dots,j_m}} \to \mathcal{O}_{F_{i_l,j_1,\dots,j_m}}$ is an $M$-graded homomorphism, the sequence (\ref{equation:toric-residue-sequence}) is actually $M$-graded. For $u \in M$, observe that:
    \[ u \in \mathcal{O}_{F_{j_1,\dots,j_m}} \iff u \in \mathcal{O}_{F_{j_i}} \text{ for all } i =1,\dots,m.\]
    Thus we can derive two things from this. First, if $u \notin \sigma^\vee \cap M$, then there are no terms in the complex (\ref{equation:toric-residue-sequence}) in degree $u$ and so, we can ignore such $u$. Second, if we fix $u \in \sigma^\vee \cap M$, and define the indexing set $I_u := \{ j \mid u \in \mathcal{O}_{F_j} \}$, then the sequence (\ref{equation:toric-residue-sequence}) in degree $u$ is just:
    \[ \bigwedge^p V \to \bigoplus_{j \in I_u} \left( \bigwedge^{p-1} L_j \otimes \frac{V}{L_j} \right) \to \bigoplus_{\substack{j_1,j_2 \in I_u,\\ j_1<j_2}} \left( \bigwedge^{p-2} L_{j_1,j_2} \otimes V_{j_1,j_2} \right) \to \dots \bigoplus_{\substack{j_1,\dots,j_p \in I_u,\\ j_1 < \dots < j_p}} V_{j_1,\dots,j_p} \to 0. \]
    But this is the same as the complex $C_{n,k_0,p}[-p]$ where $k_0 = \abs{I_u}$, and so it is an exact complex in all cohomological degrees other than $0$ by Lemma \ref{lemma:main-linear-algebra-result}. Since this holds for every $u \in M$, the sequence (\ref{equation:toric-residue-sequence}) is an exact complex at all cohomological degrees other than $0$.
\end{proof}

With the same notation as above, assume additionally that $\tau_1,\dots,\tau_k$ are \textbf{all} the rays in $\sigma$. In (\ref{equation:toric-residue-sequence}), set $V = M\otimes \mathbb{C}$ and $L_i = \tau_i^\perp$ to get the Ishida complex:
    \begin{equation}\label{equation:ishida-complex}
    \resizebox{\textwidth}{!}{$
    \begin{split}
        S_p: \quad  \bigwedge^p V \otimes \mathcal{O}_X \to \bigoplus_{j=1}^k \left( \bigwedge^{p-1} \tau_j^{\perp} \otimes \frac{V}{\tau_j^\perp} \otimes \mathcal{O}_{F_j} \right) \to \bigoplus_{1 \leq j_1<j_2 \leq k} \left( \bigwedge^{p-2} (\tau_{j_1}^\perp \cap \tau_{j_2}^\perp) \otimes \frac{V}{\tau_{j_1}^\perp} \otimes \frac{V}{\tau_{j_2}^\perp} \otimes \mathcal{O}_{F_{j_1} \cap F_{j_2}} \right) \to \dots\\
        \dots \to \bigoplus_{1 \leq j_1 < \dots < j_p \leq k} \left( \frac{V}{\tau_{j_1}^\perp} \otimes \dots \otimes \frac{V}{\tau_{j_p}^\perp} \otimes \mathcal{O}_{F_{j_1}\cap \dots \cap F_{j_p}} \right)\to 0, 
    \end{split}$}
    \end{equation}
    with $\displaystyle \bigwedge^p V \otimes \mathcal{O}_X$ lying in cohomological degree $0$. We now partially recover \cite[Theorem 3.6.3]{Oda}.
\begin{corollary}\label{corollary:simplicial-complex}
    The Ishida complex $S_p$ is exact at all cohomological degrees other than $0$.
\end{corollary}

We also get the required generalization of Lemma \ref{lemma:ses-for-Omega1Y(log E)} as a corollary.
\begin{corollary}\label{corollary:resolutioin-for-OmegapY(log E)}
    Consider the setting of Remark \ref{remark:residue-sequence-smooth-varieties}. Then, the sequence $\widetilde{C_p}$ in (\ref{equation:smooth-varieties-residue-sequence}) is a resolution of $\Omega^p_Y(\log E)$.
\end{corollary}
\begin{proof}
    First, let's check that:
    \[ 0 \to \Omega^p_Y(\log E) \to \Omega^p_Y(\log D) \to \bigoplus_{j=1}^k \Omega^{p-1}_{D_j} (\log T_j) \]
    is exact. Observe that for each $j$, we have the short exact sequence
    \[ 0 \to \Omega^p_Y(\log (D-D_j)) \to \Omega^p_Y(\log D) \to \Omega^{p-1}_{D_j} (\log T_j) \to 0. \]
    Thus
    \begin{align*}
        \ker\left(\Omega^p_Y(\log D) \to \bigoplus_{j=1}^k \Omega^{p-1}_{D_j} (\log T_j)\right) &= \bigcap_{j=1}^k \Omega^p_Y(\log (D-D_j))\\
        &= \Omega^p_Y(\log E).
    \end{align*}
    
    We now need to prove that the rest of $\widetilde{C_p}$ is exact. Choosing algebraic coordinates, we reduce to the case when $Y = \mathbb{A}^n = \Spec \mathbb{C}[y_1,\dots,y_n]$; $D_j = \{y_j=0\}$ for $j=1,\dots,k$; $E = \{y_{k+1}\dots y_r = 0\}$ with $r \leq n$.
    
    Observe that $\widetilde{C_p}$ can be identified with the complex (\ref{equation:toric-residue-sequence}) after identifying:
    \begin{align*}
        X &= \mathbb{A}^n,\\
        V &= \mathbb{C}\text{-linear span of }\frac{dy_1}{y_1}, \dots, \frac{dy_r}{y_r},dy_{r+1},\dots,dy_n, \\
        L_j &= \mathbb{C}\text{-linear span of }\frac{dy_1}{y_1},\dots,\widehat{\frac{dy_j}{y_j}}, \dots, \frac{dy_r}{y_r}, dy_{r+1}, \dots, dy_n\text{ for }j=1,\dots,k.
    \end{align*}
    By Corollary \ref{corollary:toric-residue-sequence}, the complex (\ref{equation:toric-residue-sequence}) is exact at all cohomological degrees other than $0$. Hence, $\widetilde{C_p}$ is a resolution of $\Omega^p_Y(\log E)$.
\end{proof}

\section{Proof of (1) in Theorem \ref{theorem:vanishing-in-terms-of-codim}}\label{section:proof-of-(1)}

We recall the notation in Theorem \ref{theorem:vanishing-in-terms-of-codim}. Let $X$ be an affine toric variety associated to a cone $\sigma \subset N \otimes \mathbb{R}$. Let $Y$ be the toric variety associated to a fan $\Sigma$ which refines $\sigma$, such that the induced toric morphism $\pi:Y \to X$ is a log resolution of singularities with reduced exceptional divisor $E$. Let $Z$ be the complement of the domain of $\pi^{-1}$ and let $c:=\codim_X Z$. Note that $Z$ is torus invariant since $\pi$ is a torus equivariant morphism. Let $V:= M \otimes \mathbb{C}$. Let $\tau_1,\dots,\tau_k$ be the rays of $\sigma$ and let $F_1,\dots,F_k$ be the corresponding torus invariant divisors in $X$. Since $\Sigma$ is a refinement of $\sigma$, the $1$-dimensional cones of $\Sigma$ are $\tau_1,\dots,\tau_k$ along with some additional rays $\tau_{k+1},\dots \tau_r$. Let $D_1,\dots,D_r$ be the corresponding torus invariant divisors in $Y$. Observe that $D_1,\dots,D_k$ are the strict transforms of $F_1,\dots,F_k$.

Consider the exact sequence $\widetilde{C_p}$, as in (\ref{equation:smooth-varieties-residue-sequence}), on $Y$. Observe that $T_{j_1,\dots,j_m}$ is the sum of all the torus invariant divisors of $D_{j_1} \cap \dots \cap D_{j_m}$. Therefore by Remark \ref{remark:trivialization-of-log-differentials-for-toric-varieties} applied to the smooth toric variety $D_{j_1} \cap \dots \cap D_{j_m}$:
\[\Omega^{p-m}_{D_{j_1}\cap \dots \cap D_{j_m}}(\log T_{j_1,\dots,j_m}) = \bigwedge^{p-m}(\tau_{j_1}^\perp \cap \dots \cap \tau_{j_m}^\perp) \otimes \frac{V}{\tau_{j_1}^\perp} \otimes \dots \otimes \frac{V}{\tau_{j_m}^\perp} \otimes \mathcal{O}_{D_{j_1} \cap \dots \cap D_{j_m}}. \]
In Remark \ref{remark:intersection-of-strict-transforms-toric-varieties}, we have seen that either $\pi(D_{j_1} \cap \dots \cap D_{j_m}) = F_{j_1} \cap \dots \cap F_{j_m}$ or $D_{j_1} \cap \dots \cap D_{j_m} = \emptyset$. In the first case, we get:
\begin{align*}
            \pi_*\mathcal{O}_{D_{j_1} \cap \dots \cap D_{j_m}} = \mathcal{O}_{\pi(D_{j_1} \cap \dots \cap D_{j_m})}, \quad R^i\pi_*\mathcal{O}_{D_{j_1} \cap \dots \cap D_{j_m}} = 0 \text{ for all $i>0$}.
        \end{align*}
In the second case as well, the above two conditions continue to hold for trivial reasons. Therefore, $\widetilde{C_p}$ is $\pi_*$-acyclic because
\[R^i\pi_*\Omega^{p-m}_{D_{j_1}\cap \dots \cap D_{j_m}}(\log T_{j_1,\dots,j_m}) = 0 \text{ for } i>0.\]

\begin{definition}\label{definition:residue-complex}
Define the \textbf{residue complex} $C_p$ on $X$ to be:
\[ C_p := \pi_* \widetilde{C_p}.\]
Explicitly, $C_p$ is the complex:
\begin{equation}\label{equation:residue-complex}
\resizebox{\textwidth}{!}{$
\begin{split}
    \bigwedge^p V \otimes \mathcal{O}_X \to \bigoplus_{j=1}^k \left( \bigwedge^{p-1} \tau_j^{\perp} \otimes \frac{V}{\tau_j^\perp} \otimes \mathcal{O}_{\pi(D_j)} \right) \to \bigoplus_{1 \leq j_1<j_2 \leq k} \left( \bigwedge^{p-2} (\tau_{j_1}^\perp \cap \tau_{j_2}^\perp) \otimes \frac{V}{\tau_{j_1}^\perp} \otimes \frac{V}{\tau_{j_2}^\perp} \otimes \mathcal{O}_{\pi(D_{j_1} \cap D_{j_2})} \right) \to \dots\\
    \dots \bigoplus_{1 \leq j_1 < \dots < j_p \leq k} \left( \frac{V}{\tau_{j_1}^\perp} \otimes \dots \otimes \frac{V}{\tau_{j_p}^\perp} \otimes \mathcal{O}_{\pi(D_{j_1}\cap \dots \cap D_{j_p})} \right)\to 0,
\end{split}
$}
\end{equation}
where $\displaystyle \bigwedge^p V \otimes \mathcal{O}_X$ is in cohomological degree $0$.
\end{definition}

\noindent By definition of $C_p$, since $\widetilde{C_p}$ is a $\pi_*$-acyclic resolution of $\Omega^p_Y(\log E)$, we have:
\begin{align*}
        {\bf R}\pi_*\Omega^p_Y(\log E) \simeq_{\mathrm{qis}} C_p.
    \end{align*}
    In particular:
\begin{align*}
    R^i\pi_*\Omega^p_Y(\log E) = H^i(C_p)
\end{align*}

We will now compare the residue complex $C_p$ with the Ishida complex $S_p$ (\ref{equation:ishida-complex}). In Remark \ref{remark:intersection-of-strict-transforms-toric-varieties}, we have seen that either $\pi(D_{j_1} \cap \dots \cap D_{j_m}) = F_{j_1} \cap \dots \cap F_{j_m}$ or $D_{j_1} \cap \dots \cap D_{j_m} = \emptyset$. Thus, we have a natural surjective map $S_p \twoheadrightarrow C_p$ and the kernel consists of factors involving $\mathcal{O}_{F_{j_1} \cap \dots \cap F_{j_m}}$ such that $D_{j_1} \cap \dots \cap D_{j_m} = \emptyset$.
We now make this precise.

\begin{proposition}\label{proposition:kernel-of-ishida-complex}
    For any $p \in \mathbb{N}$, define the complex $K_p := \ker(S_p \twoheadrightarrow C_p)$. Then for any $m \geq 0$, $K_p$ in cohomological degree $m$ is given by:
    \begin{equation}
        (K_p)^m = \bigoplus_{F_{j_1} \cap \dots \cap F_{j_m} \subset Z} \left( \bigwedge^{p-m}(\tau_{j_1}^\perp \cap \dots \cap \tau_{j_m}^\perp) \otimes \frac{V}{\tau_{j_1}^\perp} \otimes \dots \otimes \frac{V}{\tau_{j_m}^\perp} \otimes \mathcal{O}_{F_{j_1}\cap \dots \cap F_{j_m}} \right).
    \end{equation}
\end{proposition}
\begin{proof}
    Observe that $(K_p)^m$ precisely consists of factors involving $\mathcal{O}_{F_{j_1} \cap \dots \cap F_{j_m}}$ such that $D_{j_1} \cap \dots \cap D_{j_m} = \emptyset$. Thus, it remains to show that for any $j_1, \dots, j_m$, $D_{j_1} \cap \dots \cap D_{j_m} = \emptyset \iff F_{j_1} \cap \dots \cap F_{j_m} \subset Z$.
    
    For any $F_{j_1} \cap \dots \cap F_{j_m} \subset Z$, let us prove that $D_{j_1} \cap \dots \cap D_{j_m} = \emptyset$. 
    We claim that the face $\tau$ spanned by $\tau_{j_1}, \dots, \tau_{j_m}$ is not a cone in $\Sigma$. For if $\tau$ is a cone in $\Sigma$, then the corresponding affine open subset $U_\tau$ is an open subset of both $X$ and $Y$, and $\pi^{-1}(U_\tau) = U_\tau$, thus making $\pi$ an isomorphism over $U_\tau$. But $U_\tau \cap F_{j_1} \cap \dots \cap F_{j_m} \neq \emptyset$, which contradicts the fact that $Z$ is the smallest closed subset of $X$ outside which $\pi$ is an isomorphism. Thus, $\tau$ is not a cone in $\Sigma$, which implies that $D_{j_1} \cap \dots \cap D_{j_m} = \emptyset$.

    On the other hand, if $F_{j_1} \cap \dots \cap F_{j_m} \not\subset Z$, then $\pi$ is an isomorphism over an open subset of $F_{j_1} \cap \dots \cap F_{j_m}$, hence $D_{j_1} \cap \dots \cap D_{j_m} \neq \emptyset$.
\end{proof}

\noindent As a consequence, we get the following results about $K_p$.
\begin{corollary}\label{corollary:results-about-Kp}
    We have:
    \begin{enumerate}
        \item For $m<c$, we have $(K_p)^m = 0$.
        \item For $p \geq c$, we have $H^c(K_p) \neq 0$.
    \end{enumerate}
\end{corollary}
\begin{proof}
    For (1), observe that $\codim_X Z = c$ and so for $m<c$ and any $j_1,\dots,j_m$, we have $F_{j_1} \cap \dots \cap F_{j_m} \not\subset Z$. Therefore $(K_p)^m = 0$ for such $m$.

    For (2), first observe that $Z$ is a torus invariant subset of $X$ of codimension $c$. Hence by Remark \ref{remark:intersection-of-strict-transforms-toric-varieties}, we can find a $j_1,\dots,j_c$ such that $F_{j_1} \cap \dots \cap F_{j_c} \subset Z$. This implies that $(K_p)^c$ contains a factor involving $\mathcal{O}_{F_{j_1} \cap \dots \cap F_{j_c}}$. Now observe that $K_p$ is an $M$-graded complex. Thus, if we take
    \[ u \in (\tau_{j_1}^\perp \cap \dots \cap \tau_{j_c}^\perp \cap M) \setminus \bigcup_{l \notin \{ j_1,\dots,j_c \} } \tau_{j_1}^\perp \cap \dots \cap \tau_{j_c}^\perp \cap \tau_l^\perp  \]
    and look in degree $u$, then $((K_p)^c)_u \neq 0$ while $((K_p)^{c+1})_u = 0$. Therefore the map $(K_p)^c \to (K_p)^{c+1}$ is not injective as required.
\end{proof}

\begin{proof}[Proof of (1) in Theorem \ref{theorem:vanishing-in-terms-of-codim}]
Let us first prove (a). Assume $p < c$. By Corollary \ref{corollary:results-about-Kp}, we have that $(K_p)^m = 0$ for all $m \leq p$. Hence, $S_p \cong C_p$ and so, $C_p$ is exact at all cohomological degrees other than $0$. Therefore we have
\[ R^i\pi_*\Omega^p_Y(\log E) = 0 \text{ for $i>0$}. \]

We will now prove (b). Assume $p \geq c$. By Corollary \ref{corollary:results-about-Kp}, we have the following short exact sequence of complexes:
\[
\begin{tikzcd}
    \dots \ar[r] & 0 \ar[r] \ar[d]& (K_p)^c \ar[r] \ar[d, hook] & (K_p)^{c+1} \ar[r] \ar[d, hook] & \dots \ar[r] & (K_p)^p \ar[r] \ar[d, hook] & 0\\
    \dots \ar[r] & (S_p)^{c-1} \ar[r] \ar[d, equal]& (S_p)^c \ar[r] \ar[d, twoheadrightarrow] & (S_p)^{c+1} \ar[r] \ar[d, twoheadrightarrow] & \dots \ar[r] & (S_p)^p \ar[r] \ar[d, twoheadrightarrow] & 0\\
    \dots \ar[r] & (C_p)^{c-1} \ar[r] & (C_p)^c \ar[r] & (C_p)^{c+1} \ar[r] & \dots \ar[r] & (C_p)^p \ar[r] & 0.
\end{tikzcd}
\]
Since $S_p$ is exact at all cohomological degrees other than $0$, we get
\[ R^i\pi_*\Omega^p_Y(\log E) = H^i(C_p) = H^{i+1}(K_p) \text{ for $i>0$}. \]
Thus, we conclude that
\begin{align*}   
        &R^i\pi_*\Omega^p_{Y}(\log E) = 0 \text{, for $0< i<c-1$ and $i>p-1$,}\\
        &R^{c-1}\pi_*\Omega^p_{Y}(\log E) = H^c(K_p) \neq 0,
        \end{align*}
where the second statement follows from Corollary \ref{corollary:results-about-Kp}(2).
\end{proof}

We end with a remark regarding the higher direct images $R^i\pi_*\Omega^p_{Y}(\log E)$ when $X$ is non-simplicial.
\begin{remark}
    To summarize, the way we deduce vanishing and non-vanishing in the simplicial case is as follows. We have two complexes on $X$:
    \begin{enumerate}
        \item the residue complex $C_p$ (\ref{equation:residue-complex}) which is equal to $\mathbf{R}\pi_*\Omega^p_Y(\log E)$,
        \item the Ishida complex $S_p$ (\ref{equation:ishida-complex}) which we know has vanishing higher cohomologies.
    \end{enumerate}
    Thus, to deduce vanishing and non-vanishing of the cohomologies of $C_p$, we compare $C_p$ to $S_p$ and draw appropriate conclusions.
    
    In the non-simplicial case, we still have the complexes $C_p$ and $S_p$ and we can compare the two, but the fundamental issue is that the Ishida complex $S_p$ does not necessarily have vanishing higher cohomologies. Thus in the non-simplicial case, we need a deeper understanding of the cohomologies of the Ishida complex $S_p$ to deduce precise statements about vanishing and non-vanishing of the higher direct images $R^i\pi_*\Omega^p_{Y}(\log E)$.
\end{remark}

\textbf{Data availability.} Data sharing is not applicable to this article as no datasets were generated
or analyzed during the current study.

\textbf{Conflict of interest.} The authors have no relevant financial or non-financial interests to disclose.

\bibliographystyle{alpha}
\bibliography{reference}

\Addresses
\end{document}